\documentclass{article}
\usepackage{amssymb,amsmath,amsthm}
\usepackage{graphicx,cite}

\def\qed{\hfill {\hbox{${\vcenter{\vbox{               
   \hrule height 0.4pt\hbox{\vrule width 0.4pt height 6pt
   \kern5pt\vrule width 0.4pt}\hrule height 0.4pt}}}$}}}

\def\tr{\triangleright}

\newtheorem{theorem}{Theorem}

\newtheorem{proposition}[theorem]{Proposition}

\theoremstyle{definition}
\newtheorem{definition}{Definition}
\newtheorem{example}{Example}
\newtheorem{remark}{Remark}

\date{}

\title{\Large \textbf{Lie Ideal Enhancements of Counting Invariants}}

\author{
Gillian Roxanne Grindstaff 
 \footnote{Email: \texttt{grg02010@mymail.pomona.edu}}
\and
Sam Nelson\footnote{Email: \texttt{Sam.Nelson@cmc.edu}. Partially supported by Simons Foundation Collaboration Award 316709 }
}

\begin{document}
\maketitle

\begin{abstract}
We define enhancements of the quandle counting invariant for knots and links
with a finite labeling quandle $Q$ embedded in the quandle of units of
a Lie algebra $\mathfrak{a}$ using Lie ideals. We provide examples
demonstrating that the enhancement is stronger than the associated
unenhanced counting invariant and image enhancement invariant.
\end{abstract}



\section{Introduction} \label{I}

In the early 1980s, Joyce \cite{J} and Matveev \cite{M} 
introduced an algebraic structure called \textit{quandles} or 
\textit{distributive groupoids} with connections to knot theory. In particular
the quandle axioms can be understood as the conditions required by the 
Reidemeister moves for labelings of the arcs in a knot diagram by elements
of a quandle to correspond one to one before and after the move.

Associated
to an oriented knot or link $L$ is a \textit{fundamental quandle} $Q(L)$ which 
determines the knot group as well as the peripheral subgroup and hence is a 
complete invariant of knots up to ambient homeomorphism. Given a finite 
quandle $Q$, the set $\mathrm{Hom}(Q(L),Q)$ is a finite set (provided $L$ is
tame) and its cardinality is a computable knot invariant known as the 
\textit{quandle counting invariant}. In \cite{CJKLS} the first 
\textit{enhancement} of the quandle counting invariant, known as the 
\textit{quandle 2-cocycle} invariant, was introduced. An enhancement is a
stronger invariant which determines the counting invariant but distinguishes 
some knots or links which have the same counting invariant value. Since then,
numerous enhancements of the quandle counting invariant and its generalizations
have been explored. Each homomorphism $f\in\mathrm{Hom}(Q(L),Q)$ can be 
identified with a labeling of the arcs in a diagram of $L$ with elements
of $Q$; in particular, enhancements of the counting invariant associated to
a finite quandle $Q$ can be understood as invariants of $Q$-labeled knots.

In this paper we introduce a new enhancement of the quandle counting invariant
defied via Lie ideals in a Lie algebra, analogous to the symplectic quandle
enhancement defined in \cite{NN}. Our construction involves embedding a finite
quandle $Q$ in the quandle of units $Q(\mathfrak{a})$ of the universal 
enveloping algebra $A$ of a Lie algebra $\mathfrak{a}$, then using the Lie 
algebra structure to 
obtain an invariant signature for each quandle coloring of our knot or link 
$K$. Different embeddings of the same finite quandle into Lie algebras in 
general yield different enhancements. We can think of the embedding 
$Q\to Q(\mathfrak{a})$  as a kind of knotting analogous to embedding
copies of $S^1$ into $S^3$; then these invariants are defined via pairs of
``knotted quandles'' 
philosophically similar to the approach to finite type invariants in 
\cite{GVP}, where a knot is compared to another ``knot'' via evaluation of 
a bilinear form. 

The paper is organized as follows. In Section
\ref{QB} we review the basics of quandles and the quandle counting invariant,
as well as some previously studied enhancements. In Section \ref{L} we 
introduce the Lie Ideal Enhanced polynomial invariant or LIE polynomial 
associated to Lie algebra with finite quandle of units. In Section \ref{CE} we 
collect computations and examples of the new invariants, and in Section \ref{Q}
we conclude with some questions for future research.

\section{Quandles}\label{QB}

We begin with a definition (see \cite{FR,J,M}).
\begin{definition}\textup{
A \textit{quandle} is a set $Q$ with binary operations 
$\tr,\tr^{-1}:Q\times Q\to Q$ satisfying for all $x,y,z\in Q$
\begin{itemize}
\item[(i)] $x\tr x=x$,
\item[(ii)] $(x\tr y)\tr^{-1} y=x=(x\tr^{-1} y)\tr y$, and
\item[(iii)] $(x\tr y)\tr z=(x\tr z)\tr (y\tr z)$.
\end{itemize}
If we have only (ii) and (iii), $Q$ is a \textit{rack}. 
}\end{definition} 

\begin{example}\textup{
Let $G$ be a group. Then for each $n\in\mathbb{Z}$, $G$ is a quandle under
$n$-fold conjugation
\[x\tr y = y^{-n}xy^n,\quad x\tr^{-1} y = y^{n}xy^{-n}.\]
We usually denote this quandle structure as $\mathrm{Conj}_n(G)$.
}\end{example}

\begin{example}\textup{
Let $\Lambda=\mathbb{Z}[t^{\pm 1}]$ and let $A$ be any $\Lambda$-module.  
Then, $A$ is a quandle under the operations
\[x\tr y = tx+(1-t)y,\quad x\tr^{-1} y = t^{-1}x+(1-t^{-1})y.\]
Quandles of this type are known as \textit{Alexander quandles}.
}\end{example}

\begin{example}\textup{
Let $\mathbb{F}$ be a field and $V$ an $\mathbb{F}$-vector space with
symplectic form $\langle\cdot,\cdot\rangle:V\times V\to \mathbb{F}$.
Then, $V$ is a quandle under the operations
\[\vec{x}\tr \vec{y} = \vec{x}+\langle\vec{x},\vec{y}\rangle \vec{y},\quad
\vec{x}\tr^{-1} \vec{y} = \vec{x}-\langle\vec{x},\vec{y}\rangle \vec{y}.\]
Quandles of this type are known as \textit{symplectic quandles}.
}\end{example}

More generally, if $X=\{1,2,\dots, n\}$ then we can define a quandle structure
on $X$ by $i\tr j=M_{ij}$ with an $n\times n$ matrix $M_X$ with the properties 
that $M_{ii}=i$, each column is a permutation of $X$, and 
$M_{M_{ij}k}=M_{M_{ik}M_{jk}}$ for $i,j,k\in X$. We call $M_X$ the \textit{quandle 
matrix} of $X$.

\begin{example}\textup{
The Alexander quandle $X=\Lambda/(3,t-2)=\mathbb{Z}_3[t]/(t-2)$ has quandle 
matrix
\[M_X=\left[\begin{array}{ccc}
1 & 3 & 2 \\ 
3 & 2 & 1 \\ 
2 & 1 & 3
\end{array}\right].\]}
\end{example}

\begin{example}\textup{
The \textit{fundamental quandle} of a link $L$ with diagram 
$D$ has a generator for each arc in $D$ with relation $z=x\tr y$ at each 
crossing. Elements of $Q(L)$ are then equivalence classes of quandle
words in these generators modulo the equivalence relation generated by the
quandle axioms and the crossing relations.
\[\includegraphics{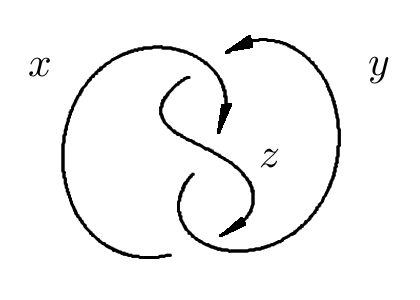} \quad \raisebox{0.5in}{
$Q(L)=\langle x,y,x\ | x\tr y=z, y\tr z=x, z\tr x=y \rangle$.}\]
}\end{example}

\begin{example}\textup{Let $L$ be an oriented link in $S^3$ and $N(L)$ a regular
neighborhood of $L$. Geometrically, the fundamental quandle of $L$ is the set 
of homotopy classes of paths in the link complement $S^3\setminus N(L)$ from a 
base point $\ast\in S^3\setminus N(L)$ to $N(L)$ where the endpoint is allowed 
to wander along $N(L)$ during the homotopy. For each such path $y$ there is 
is a canonical meridian $m(y)$  based at the terminal point of $y$ and
linking the core of the component of $N(L)$ once. The quandle operation 
is then given by the path along $y$, around $m(y)$, backwards along $y$, then
along $x$  
\[\includegraphics{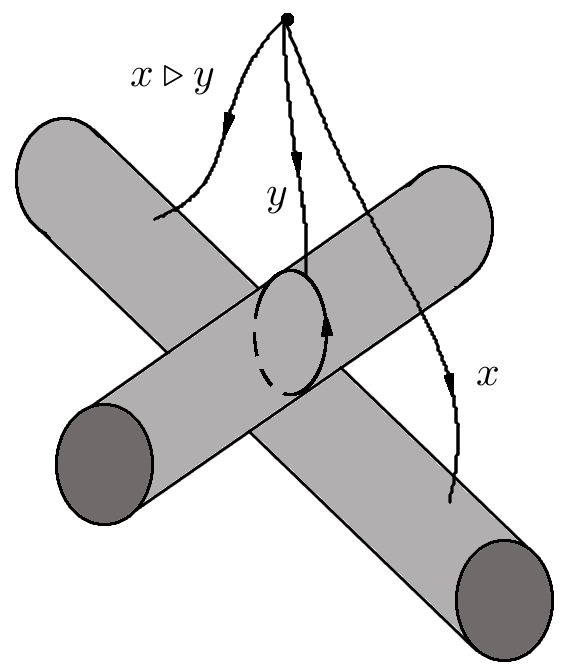} \quad\raisebox{1in}{$x\tr y=xy^{-1}m(y)^{-1}y$}\]
as depicted. See \cite{FR,J,M} for more.
}\end{example}

For each $y\in Q$, define $f_y:Q:\to Q$ by $f_y(x)=x\tr y$. Then quandle axiom
(ii) says that $f_y$ is invertible for each $y$; in particular, the operation
$\tr$ determines the operation $\tr^{-1}$. A subset $S\subset Q$ of a quandle 
$Q$ which is closed under $\tr$ and $\tr^{-1}$ is a \textit{subquandle} of $Q$;
if $Q$ is finite, then closure under $\tr$ implies closure under $\tr^{-1}$. 
A map 
$f:Q\to Q'$ between quandles $Q$ and $Q'$ with operations $\tr$ and $\tr'$
is a \textit{homomorphism of quandles} if for all $x,y\in Q$ we have
\[f(x\tr y)=f(x)\tr' f(y).\]

The quandle axioms can be understood as algebraically encoding the 
Reidemeister moves. More precisely, a \textit{quandle coloring} of an oriented 
knot diagram $K$ by a quandle $Q$ is an assignment of an element of $Q$ to each
arc in $K$ such that at each crossing, the result of an arc labeled $x$ crossing
under and arc labeled $y$ from right to left is an arc labeled $x\tr y$.
\[\includegraphics{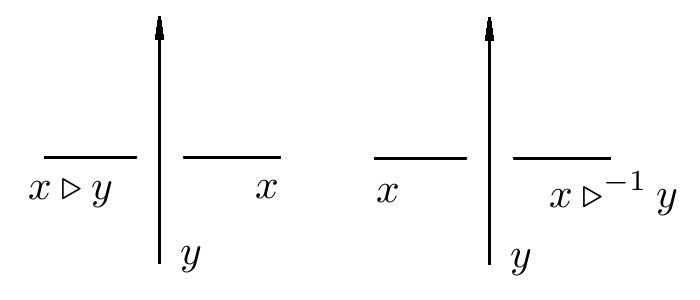}\]
In particular, the quandle axioms are the conditions required to ensure that
for each quandle coloring of a knot diagram before a move, there is a unique
corresponding quandle coloring of the diagram after the move. 
\[\includegraphics{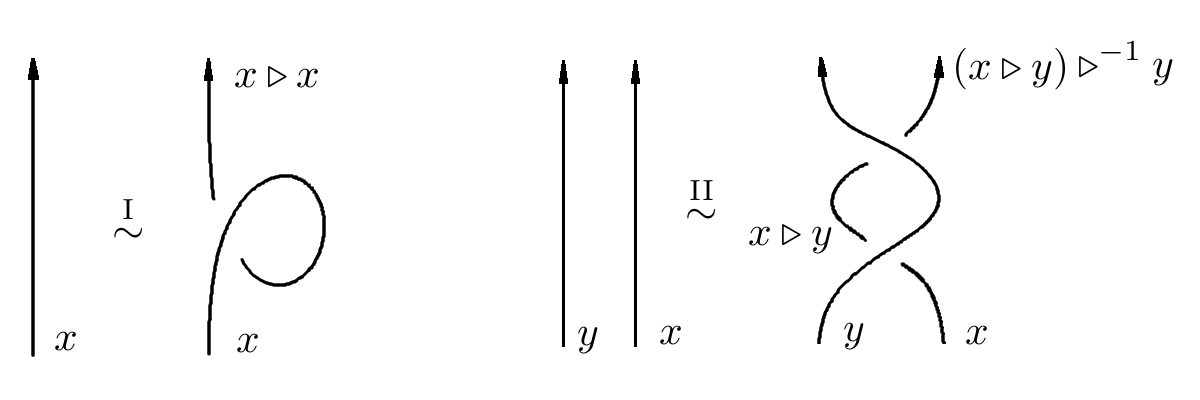}\]
\[\includegraphics{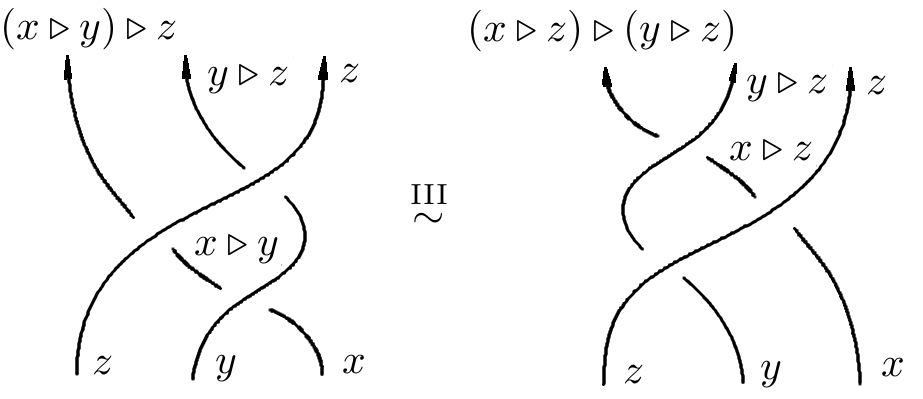}\]
Then we have:

\begin{theorem}
Let $Q$ be a finite quandle and $L$ an oriented knot or link. Then the
number $|\mathrm{Hom}(Q(K),Q)|$ of quandle colorings of a diagram of $L$ is
invariant under Reidemeister moves. 
\end{theorem}

A quandle coloring of a knot or link $L$ by a quandle $Q$ determines a unique 
homomorphism of quandles $f:Q(L)\to Q$ by assigning an image in $Q$ to each 
generator of the fundamental quandle $Q(L)$; then the homomorphism condition 
requires that $f(x\tr y)=f(x)\tr f(y)$ for each $x,y\in Q(L)$. In particular,
an assignment of elements of $Q$ to arcs in a diagram of $L$ determines a
quandle homomorphism $f:Q(L)\to Q$ if and only if the crossing relations
are satisfied at every crossing, and the homomorphism so determined is unique.
The number of quandle colorings $|\mathrm{Hom}(Q(L),Q)|$ of a knot or link
$L$ is called the \textit{quandle counting invariant} of $K$ with coloring
quandle $Q$. If $Q$ is finite, then $|\mathrm{Hom}(Q(L),Q)|$
is a positive integer greater than or equal to $|Q|$, since monochromatic 
colorings are always valid. 

Now let $\phi$ be a an invariant of $Q$-colored knot diagrams. Instead of 
counting ``1'' for each element of $\mathrm{Hom}(Q(L),Q)$, we can collect
$\phi(f)$ values for each $f\in\mathrm{Hom}(Q(L),Q)$ to obtain a multiset
\[\Phi^{\phi}_Q(K)=\{\phi(f)\ |\ f\in\mathrm{Hom}(Q(L),Q) \}\] whose cardinality is the quandle counting invariant. We
call such an invariant an \textit{enhancement} of the counting invariant.
Especially for integer-valued enhancements, we often replace the multiset 
with its generating function to get a polynomial invariant for ease of 
comparison. For instance, the multiset $\{0,0,1,2,2,2\}$ corresponds to the
polynomial $2+u+3u^2$.

\begin{example}\textup{
Let $Q$ be a finite quandle and $L$ a tame oriented knot or link. For each
$f\in\mathrm{Hom}(Q(L),Q)$, let $\phi(f)$ be the cardinality of the image 
subquandle $\phi(f)=|\mathrm{Im}(f)|$. Then the \textit{image enhancement}
is the multiset 
\[\Phi^{\mathrm{Im},M}_Q(L)=\{|\mathrm{Im}(f)|\ :\ f\in\mathrm{Hom}(Q(L),Q)\}\]
or the polynomial
\[\Phi^{\mathrm{Im}}_Q(L)=\sum_{f\in\mathrm{Hom}(Q(L),Q)} u^{|\mathrm{Im}(f)|}.\]
For instance, the trefoil knot $3_1$ has three monochromatic colorings by the 
quandle $Q=\Lambda_3/(t-2)$ and six surjective colorings, so we have
quandle counting invariant $|\mathrm{Hom}(Q(K3_1),Q)|=9$ and image enhancement
$\Phi^{\mathrm{Im}}_Q(3_1)=3u+6u^3$. See \cite{N} for more.
}\end{example}

\section{Lie Ideal Enhancements}\label{L}

Let $R$ be a commutative ring, $A$ be an associative unital $R$-algebra 
and let $\mathfrak{a}$ be the associated Lie algebra (that is, $\mathfrak{a}$ 
is $A$ endowed with the Lie bracket $[u,v]=uv-vu$). Then for any integer 
$n\in \mathbb{Z}$, the group of units of 
$A$, $A^*=\{u\in A\ |\ \exists v\in A \ s.t. \ uv=vu=1 \}$, is a quandle under 
the $n$-fold conjugation operation
\[u\tr v=v^{-n}uv^n=u+v^{-n}[u,v^n]\]
which we call the \textit{$n$th quandle of units} of $\mathfrak{a}$, denoted
$Q_n(\mathfrak{a})$. If $\mathfrak{a}$ is finite, then so is 
$Q_n(\mathfrak{a})$, though the converse is not necessarily true. If $n=1$,
we will write $Q(\mathfrak{a})$ instead of $Q_1(\mathfrak{a})$. 

\begin{example}
Let $\mathbb{F}$ be a field. Then for any positive integer $m$, the matrix 
algebra $M_m(\mathbb{F})$ of square $m\times m$ matrices with entries in $F$
has quandle of units $GL_m(\mathbb{F})$ consisting of invertible matrices.
In particular, if $m=1$ then the quandle of units is the conjugation
quandle of the abelian group $\mathbb{F}\setminus 0$ and hence is
a trivial quandle.
\end{example}

\begin{example}
Let $G$ be a group and $\mathbb{F}$ a field. Then the group 
$\mathbb{F}$-algebra
\[\mathfrak{a}=\mathbb{F}[G]=\left\{\sum_{g\in G}\alpha_g g\ |\ \alpha_g\in \mathbb{F},\ g\in G\right\}\]
has quandle of units $Q(\mathfrak{a})$ containing a subquandle isomorphic to 
$\mathrm{Conj}(G)$. 
\end{example}

In some cases this quandle of units coincides with symplectic quandles as 
defined in \cite{NN}. More precisely, we have the following:

\begin{theorem}\label{thm2}
Let $A$ be an associative unital $\mathbb{F}$-algebra with 
symplectic form $\langle\cdot,\cdot\rangle:A\times A\to \mathbb{F}$. If
the symplectic quandle structure on $A^*$ agrees with the quandle of units
$Q(\mathfrak{a})$, then $[\vec{v}^{-1},\vec{u}] = [\vec{v},\vec{u}^{-1}]$ for all
$\vec{u},\vec{v}\in A^{\ast}$.
\end{theorem}

\begin{proof}
For the symplectic quandle structure on $A^*$ we have 
$ \vec{u} \tr \vec{v}  = \vec{u} + \langle \vec{u},\vec{v}\rangle \vec{v},$ 
where $\langle\cdot,\cdot \rangle$ is antisymmetric. If 
$\vec{u} \triangleright \vec{v}$ is the same as the quandle action 
defining $Q(\mathfrak{a})$ above, then we have
\[\langle \vec{u},\vec{v}\rangle\vec{v}=\vec{v}^{-1}[\vec{u},\vec{v}]\]
and thus
\[
\langle \vec{u},\vec{v}\rangle\vec{1}
=\langle \vec{u},\vec{v}\rangle\vec{v}\vec{v}^{-1}
=\vec{v}^{-1}[\vec{u},\vec{v}]\vec{v}^{-1}
=\vec{v}^{-1}(\vec{u}\vec{v}-\vec{v}\vec{u})\vec{v}^{-1}
=\vec{v}^{-1}{u}-\vec{u}\vec{v}^{-1}
=[\vec{v}^{-1},u].
\]
Now, we also have
\[
\langle \vec{u},\vec{v}\rangle\vec{1}
=-\langle \vec{v},\vec{u}\rangle\vec{1}
=-[\vec{u}^{-1},\vec{v}] = [\vec{v},\vec{u}^{-1}]
\]
and thus $[\vec{v}^{-1},\vec{u}] = [\vec{v},\vec{u}^{-1}]$.


\end{proof}

\begin{proposition}
The quandle of units in a Lie algebra $A$ over a field $\mathbb{F}$
is involutory if and only if either (1) $\mathbb{F}$ has characteristic 2 or
(2) the group of units $A^{\ast}$ is abelian.
\end{proposition}

\begin{proof}
Recall that a quandle $Q$ is involutory if we 
have $\vec{u}\tr\vec{v}=\vec{u}\tr^{-1}\vec{v}$ for all $\vec{u},\vec{v}\in Q$.
In the quandle of units case, we have
\[\vec{u}\tr \vec{v}= \vec{u}+\vec{v}^{-1}[\vec{u},\vec{v}]\]
and
\[\vec{u}\tr^{-1} \vec{v}= \vec{u}-\vec{v}^{-1}[\vec{u},\vec{v}].\]
Then $Q(\mathfrak{a})$ is involutory iff
\[[\vec{u},\vec{v}]=-[\vec{u},\vec{v}]=[\vec{v},\vec{u}],\]
that is, iff 
\[\vec{u}\vec{v}-\vec{v}\vec{u}=\vec{v}\vec{u}-\vec{u}\vec{v}\]
i.e., iff 
\[2\vec{u}\vec{v}=2\vec{v}\vec{u}\]
for all $\vec{u},\vec{v}\in A^{\ast}$. If $\mathbb{F}$ has characteristic $2$,
this is automatic; otherwise, $Q(\mathfrak{a})$ is involutory 
iff $\vec{u}\vec{v}=\vec{v}\vec{u}$ for all $\vec{u},\vec{v}\in A^{\ast}$.
\end{proof}

Now, if $L$ is a knot or link and $Q\subset Q(\mathfrak{a})$ is a finite
subquandle of the quandle of units of a Lie algebra $\mathfrak{a}$, 
then the set of quandle homomorphisms $\mathrm{Hom}(Q(L),Q)$ is a finite set 
which is unchanged by Reidemeister moves. In particular, we can use the 
structure of $\mathfrak{a}$ to obtain invariant signatures of the quandle 
labelings of $L$ by $Q$ to enhance the quandle counting invariant.

\begin{definition}\textup{
Let $L$ be a link, $\mathfrak{a}$ a finite dimensional 
Lie algebra and $Q\subset Q(\mathfrak{a})$ a finite subquandle of the
quandle of units $Q(\mathfrak{a})$. Then the multiset 
\[\Phi_{Q,\mathfrak{a}}^{M}(L)=\{I(\mathrm{Im}(f))\ |\ f\in\mathrm{Hom}(Q(K),
Q)\}\]
of Lie ideals $I(\mathrm{Im}(f))$ generated by the image subquandles 
$\mathrm{Im}(f)$ for $f\in\mathrm{Hom}(Q(K),Q)$
is the \textit{Lie Ideal Enhancement} of the quandle counting invariant.
The polynomial
\[\Phi_{Q,\mathfrak{a}}(L)=
\left\{\begin{array}{ll} \displaystyle
\sum_{f\in\mathrm{Hom}(Q(L),Q)}u^{|I(\mathrm{Im}(f))|} & \mathfrak{a} \ \mathrm{finite} \\
\displaystyle \sum_{f\in\mathrm{Hom}(Q(L),Q)}
u^{\mathrm{rank}(I(\mathrm{Im}(f)))} & \mathfrak{a}\ \mathrm{infinite} \\ \end{array}\right.\]
is the \textit{Lie Ideal Enhancement polynomial} of $L$ with respect to the
the subquandle $Q\subset Q(\mathfrak{a})$ of the quandle of units of the
Lie algebra $\mathfrak{a}$. If $Q=Q(\mathfrak{a})$, we will denote 
$\Phi_{Q,\mathfrak{a}}$ simply as $\Phi_{\mathfrak{a}}$.
}\end{definition}

By construction, we have the following theorem:

\begin{theorem}
For any Lie algebra $\mathfrak{a}$, the multiset $\Phi_{Q,\mathfrak{a}}^{I}(L)$
and polynomial $\Phi_{Q,\mathfrak{a}}(L)$ are link invariants.
\end{theorem}

We can also define enhancements using the associative ideals 
$AI(\mathrm{Im}(f))$ in the enveloping algebra $A$
\[\Phi_{Q,A}^{M}(L)=\{AI(\mathrm{Im}(f))\ |\ f\in\mathrm{Hom}(Q(K),Q)\}\]
and 
\[\Phi_{Q,A}(L)=\sum_{f\in\mathrm{Hom}(Q(L),Q)}u^{|AI(\mathrm{Im}(f))|}\]
and combine these to get a two-variable polynomial
\[\Phi_{Q,A,\mathfrak{a}}(L)=\sum_{f\in\mathrm{Hom}(Q(L),Q)}u^{|I(\mathrm{Im}(f))|}v^{|AI(\mathrm{Im}(f))|}.
\]

\begin{remark}
These invariants are defined for virtual knots as well as classical knots
via the usual method of ignoring virtual crossings; see \cite{K} for more
about virtual knot invariants defined from quandles.
\end{remark}

\section{Computations and Examples}\label{CE}

In this section we collect a few examples of LIE invariants and their 
computation. Our computations use custom \texttt{Python} code available for
download from the second author's website at \texttt{www.esotericka.org}.

\begin{example}\textup{If $A=\mathbb{Z}[G]$ where $G$ is an abelian group,
then $Q(\mathfrak{a})$ is a trivial quandle, i.e. $x\tr y=x$ for all $x,y\in
Q(\mathfrak{a})$. If $L$ is a link of $c$ components, then each component
is monochromatic in any quandle coloring, and every coloring assigning the same
fixed colors to the arcs comprising a component is valid. Hence, there are
$|Q(\mathfrak{a})|^c$ colorings of $L$ by such a quandle. Since the Lie bracket
is zero for $G$ abelian, the Lie ideals are just the subspaces generated by the
quandle elements. Such a coloring using $k$ distinct colors will then generate 
a Lie ideal of dimension $k$. Let $n=|Q(\mathfrak{a})|$; then for each 
$k\in\{1,2,\dots,c\}$ there are $(n)_kS(c,k)$ such colorings where 
$(n)_k=n(n-1)\dots(n-k+1)$ is the $k$th falling factorial of $n$ and $S(c,k)$
is the Stirling number of the second kind, i.e. the number of ways of 
partitioning a set of cardinality $c$ into $k$ nonempty subsets. For instance,
if $\mathfrak{a}=\mathbb{Z}_3[C_4]$ where $C_4=\{1,x,x^2,x^3 \ | x^4=1\}$ then
$Q(\mathfrak{a})$ is the trivial quandle of $n=4$ elements; if $L$ is a link
of $c=3$ components, then $S(3,1)=1$, $S(3,2)=3$ and $S(3,3)=1$, and we have
\[\Phi_{\mathfrak{a}}(L)=(4)_1S(3,1)u^3+(4)_2S(3,2)u^9+(4)_3S(3,3)u^{27}=
4u^3+36u^9+24u^{27}.\]
}\end{example}

In particular, the previous example implies the following:
\begin{theorem}
Let $L$ be a link of $c$ components and $\mathfrak{a}$ an abelian Lie algebra 
over a field $\mathbb{F}$ with trivial quandle of units $Q(\mathfrak{a})$ of 
cardinality $n$. Then the LIE polynomial is given by
\[\Phi_{\mathfrak{a}}(L)=\sum_{k=1}^{c} (|Q(\mathfrak{a})|)_kS(c,k)u^{p^k}\]
if $|\mathbb{F}|=p$ or
\[\Phi_{\mathfrak{a}}(L)=\sum_{k=1}^{c} (|Q(\mathfrak{a})|)_kS(c,k)u^k\]
if $|\mathbb{F}|$ is not finite.
\end{theorem}

Thus, to get nontrivial invariant values, we must (unsurprisingly) focus on 
non-abelian Lie algebras.

\begin{example}
Let $\mathfrak{a}=M_2(\mathbb{Z}_2)$, the Lie algebra of $2\times 2$ matrices 
with entries in $\mathbb{Z}_2$. Then the quandle of units of $\mathfrak{a}$ is 
isomorphic to the conjugation quandle of $S_3=\langle\alpha,\beta \ |\ \alpha^2=\beta^3=1,\ (\alpha\beta)^2=1\rangle$ with
\[\alpha=\left[\begin{array}{rr}0 & 1\\ 1 & 0 \end{array}\right]
\quad\mathrm{and}\quad
\beta=\left[\begin{array}{rr}1 & 1 \\ 1 & 0\end{array}\right].\] 
The link $L7a4$ has 30 quandle
labellings by $Q(\mathfrak{a})$ including for example
\[\includegraphics{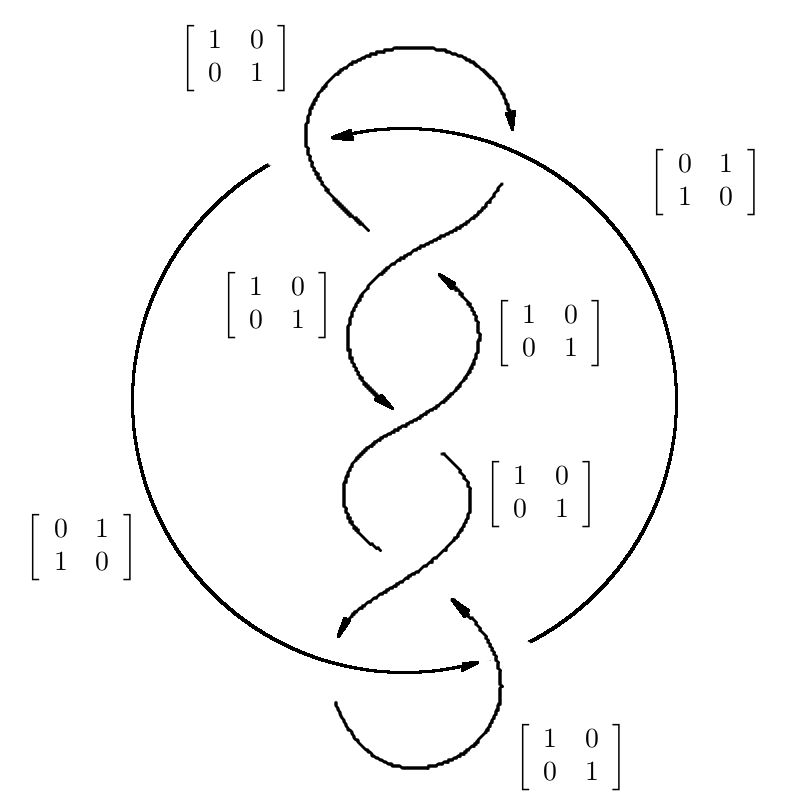}\]
The subquandle of $Q(\mathfrak{a})$ generated by 
$\left[\begin{array}{rr}
1 & 0 \\
0 & 1 \\
\end{array}\right]
$ and $
\left[\begin{array}{rr}
0 & 1 \\
1 & 0 \\
\end{array}\right]
$
is a trivial quandle of just those two elements; the Lie ideal they generate
consists of the four matrices
\[\left\{
\left[\begin{array}{rr}
0 & 0 \\
0 & 0 \\
\end{array}\right], \
\left[\begin{array}{rr}
1 & 0 \\
0 & 1 \\
\end{array}\right],\
\left[\begin{array}{rr}
0 & 1 \\
1 & 0 \\
\end{array}\right],\
\left[\begin{array}{rr}
1 & 1 \\
1 & 1 \\
\end{array}\right]
\right\}\]
so this quandle coloring contributes $u^4$ to $\Phi_{\mathfrak{a}}(L7a4)$.
\end{example}

\begin{remark}
In the previous example, the constant coloring  of the link by the matrix
$\left[\begin{array}{rr}
0 & 1 \\
1 & 0 \\
\end{array}\right]$
generates a four-element Lie ideal, but spans only a 2 element subspace of
$\mathfrak{a}$ considered as a symplectic vector space. Thus, this constant 
coloring contributes $u^4$ to the LIE invariant and $u^2$ to the symplectic
quandle enhancement, so we expect the symplectic enhancement and LIE invariant
to contain different information. Indeed, the follwing examples show that the
LIE invariant is not determined by the quandle counting invariant, the image 
enhancement invariant or the symplectic quandle enhancement (in case
the quandle of units is symplectic).
\end{remark}

\begin{example}\label{ex1}
The links $L7a4$ and $L7n1$ both have counting invariant value $30$ with
respect to the quandle of units $Q(\mathfrak{a})$ from the previous example
but have distinct Lie ideal enhanced polynomials of 
$\Phi_{\mathfrak{a}}(L7a4)=20u^{16}+9u^4+u^2$ and
$\Phi_{\mathfrak{a}}(L7n1)=14u^{16}+6u^8+9u^4+u^2$. In particular, this example
shows that $\Phi_{\mathfrak{a}}$ is a proper enhancement of the counting invariant.
\[\begin{array}{cc}
\includegraphics{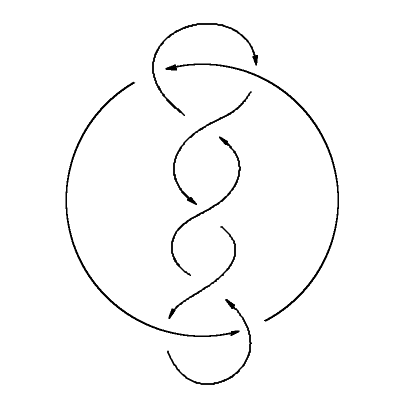} & \includegraphics{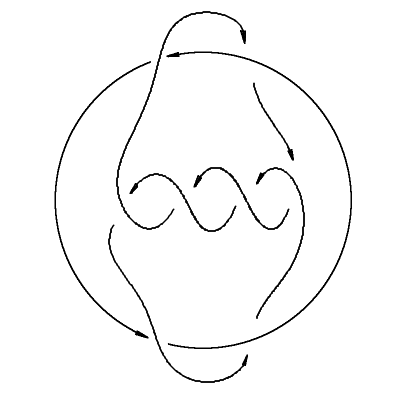} \\
\Phi_{\mathfrak{a}}(L7a4)=20u^{16}+9u^4+u^2 &
\Phi_{\mathfrak{a}}(L7n1)=14u^{16}+6u^8+9u^4+u^2 \\
\end{array}\]
\end{example}

\begin{example}
Example \ref{ex1} shows that the LIE invariant is not determined by the 
quandle counting invariant; however, the two links listed are also
distinguished by the image enhancement invariant, with 
$\Phi_{\mathfrak{a}}^{\mathrm{Im}}(L7a4)=12u^5+12u^2+6u$ and
$\Phi_{\mathfrak{a}}^{\mathrm{Im}}(L7n1)=6u^5+6u^4+12u^2+6u$.
The virtual links $L_1$ and $L_2$ below have the same quandle counting invariant
and image enhancement values for the two-fold conjugation quandle of $S^3$
but are distinguished by the LIE invariant for $\mathfrak{a}=\mathbb{Z}_2[S_3]$
and $Q=Q_2(\mathfrak{a})$:
\[\includegraphics{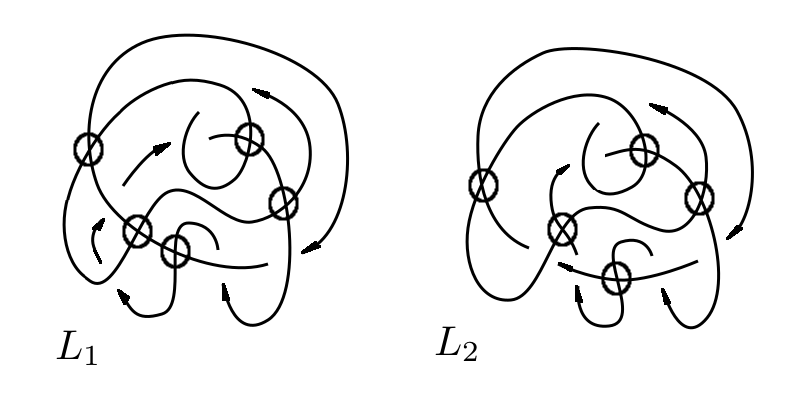}\]
\[\begin{array}{cc}
\Phi_{Q_2(\mathfrak{a})}^{\mathbb{Z}}(L_1)= 126&
\Phi_{Q_2(\mathfrak{a})}^{\mathbb{Z}}(L_2)= 126\\
\Phi_{Q_2(\mathfrak{a})}^{\mathrm{Im}}(L_1)= 24u^5 + 12u^4 + 30u^3 + 54u^2 + 6u &
\Phi_{Q_2(\mathfrak{a})}^{\mathrm{Im}}(L_2)= 24u^5 + 12u^4 + 30u^3 + 54u^2 + 6u\\
\Phi_{Q_2(\mathfrak{a}),\mathfrak{a}}(L_1)=12u^{64} + 60u^{32} + 50u^{16} + 3u^8 + u^2 &
\Phi_{Q_2(\mathfrak{a}),\mathfrak{a}}(L_2)= 18u^{64} + 54u^{32} + 50u^{16} + 3u^8 + u^2\\
\end{array}\]
Hence, the LIE invariant is not determined by the image enhancement or
the quandle counting invariant.
\end{example}

\begin{example}
Let $\mathfrak{a}=M_2(\mathbb{Z}_2)$. Then the virtual links below are not 
distinguished by the quandle counting invariant, image enhancement or 
symplectic quandle enhancement invariant, but are distinguished by the LIE 
invariant:
\[\includegraphics{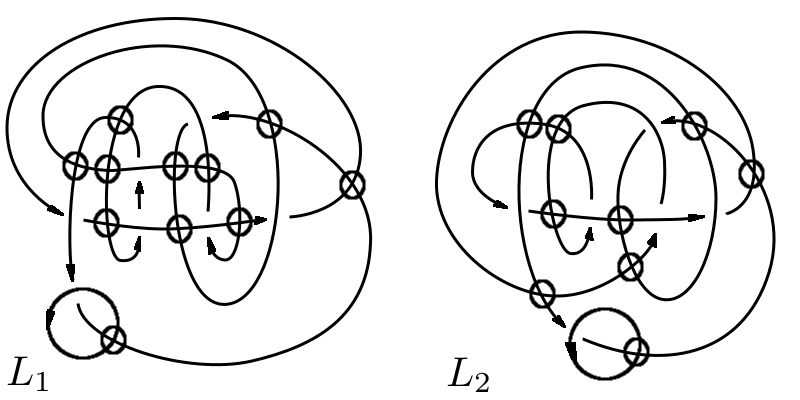}\]
\[\begin{array}{cc}
\Phi_{Q(\mathfrak{a})}^{\mathbb{Z}}(L_1)= 108 &
\Phi_{Q(\mathfrak{a})}^{\mathbb{Z}}(L_2)= 108\\
\Phi_{Q(\mathfrak{a})}^{\mathrm{Im}}(L_1)= 12u^6 + 30u^5 + 12u^4 + 12u^3 + 36u^2 + 6u &
\Phi_{Q(\mathfrak{a})}^{\mathrm{Im}}(L_2)= 12u^6 + 30u^5 + 12u^4 + 12u^3 + 36u^2 + 6u \\
\Phi_{Q(\mathfrak{a})}^{\mathrm{Symp}}(L_1)= 30u^{16} + 24u^8 + 48u^4 + 6u^2 &
\Phi_{Q(\mathfrak{a})}^{\mathrm{Symp}}(L_2)= 30u^{16} + 24u^8 + 48u^4 + 6u^2 \\
\Phi_{Q(\mathfrak{a}),\mathfrak{a}}(L_1)=74u^{16} + 12u^8 + 21u^4 + u^2 &
\Phi_{Q(\mathfrak{a}),\mathfrak{a}}(L_2)= 68u^{16} + 18u^8 + 21u^4 + u^2\\
\end{array}\]
\end{example}

\begin{example}\textup{
For our final example we chose a quandle embedded in a relatively small Lie
algebra for speed of computation and computed the LIE polynomial invariant for 
all prime classical links with up to seven crossings. The results are 
collected in the table.
\[M_Q=\left[\begin{array}{ccccc}
1 & 3 & 1 & 3 & 3 \\
4 & 2 & 5 & 5 & 4 \\
3 & 1 & 3 & 1 & 1 \\
5 & 5 & 2 & 4 & 2 \\
2 & 4 & 4 & 2 & 5
\end{array}\right] \quad \mathfrak{a}=\mathbb{Z}_2[S_3]\quad
\begin{array}{r|l}
\Phi_{Q,\mathfrak{a}}(L) & L \\ \hline
4u^{16}+3u^8 & L2a1, L6a2, L7a6 \\ 
8u^{16}+3u^8 & L6a4 \\ 
10u^{16}+3u^8 & L6a3, L7a5 \\
6u^{32}+4u^{16}+3u^8 & L7a2, L7a3, L7n1, L7n2 \\ 
12u^{32}+4u^{16}+3u^8 & L4a1, L5a1, L7a4\\ 
12u^{32}+10u^{16}+3u^8 & L6a1, L7a1 \\ 
18u^{32}+8u^{16}+3u^8 & L6n1, L7a7 \\
18u^{32}+14u^{16}+3u^8 & L6a5 \\ 
\end{array}
\]
}\end{example}


\section{Questions For Future Work}\label{Q}

We close with a few questions and direction for future research.

In \cite{FR}, it is noted that the exponential map satisfies the augmented 
rack identity, giving a Lie algebra the structure of an augmented rack with 
augmentation group the associated Lie group. What are sufficient
conditions for the resulting rack to be a finite quandle? In the case of
infinite Lie quandles, can Lie ideals be used to enhance the topological
version of the counting invariant defined in \cite{R} or to define power
series-valued counting invariants and enhancements? 
What are some other quandle structures on Lie algebras?

\bibliography{gg-sn-rev1}{}
\bibliographystyle{abbrv}

%
%
%
%
%
%
%
%

\bigskip

\noindent\textsc{Department of Mathemtical Sciences\\
Claremont McKenna College \\
850 Columbia Ave. \\
Claremont, CA 91767
}

\end{document}